\documentclass{scrartcl}
\usepackage[margin=2.5cm,bottom=1cm,includefoot]{geometry}
\usepackage{amsmath,amssymb,amsthm}
\usepackage{bbm}
\usepackage{array}
\usepackage{mathrsfs}
\usepackage{enumerate}
\usepackage{graphicx}
\usepackage[utf8]{inputenc}
\usepackage{hyperref}
\usepackage[all]{hypcap}
\usepackage{float}

\usepackage{color}

\newtheorem{thm}{Theorem}
\newtheorem{lem}[thm]{Lemma}
\newtheorem{prop}[thm]{Proposition}
\newtheorem{cor}[thm]{Corollary}
\newtheorem{conj}[thm]{Conjecture}
\theoremstyle{remark}
\newtheorem{rem}[thm]{Remark}

\let\BFseries\bfseries\def\bfseries{\BFseries\mathversion{bold}}

\newcommand{\N}{\mathbb{N}}
\newcommand{\R}{\mathbb{R}}
\newcommand{\E}{\mathbb{E}}
\newcommand{\p}{\mathbb{P}}
\newcommand{\CC}{\mathcal{C}}

\newcommand{\BB}{\mathcal{B}}
\newcommand{\1}{\mathbbm{1}}
\newcommand{\eps}{\theta}
\DeclareMathOperator{\e}{e}
\newcommand{\D}{\mathrm{d}}
\newcommand{\op}{\operatorname}
\newcommand{\lne}{<}
\newcommand{\gne}{>}
\newcommand{\abs}{|}
\newcommand{\GG}{G}

\begin{document}
\title{Brownian motion conditioned to have\\ restricted $L_2$-norm}
\author{Frank Aurzada\footnote{Technical University of Darmstadt}\and Mikhail Lifshits\footnote{St.~Petersburg State University} \and Dominic T. Schickentanz\footnotemark[1]$\ ^{,}$\footnote{Paderborn University}}
\date{\today}
\maketitle
\begin{abstract}
We condition a Brownian motion on having an atypically small $L_2$-norm on a long time interval. The obtained limiting process is a non-stationary Ornstein--Uhlenbeck process.
\end{abstract}

\section{Introduction and main results}

This paper is concerned with stochastic processes under constraints. Starting with~\cite{doob57}, where Doob conditioned a Brownian motion not to enter the negative half-line, many other constraints have been studied. In order to motivate the present results, recall the following classical situation. In Theorem~3.1 of~\cite{Kni69}, Knight conditioned a Brownian motion~$B$ on not leaving a bounded interval, say~$[-\eps,\eps]$:

\begin{thm} \label{thmUni}
Let~$\eps \gne 0$ and let~$W=(W_t)_{t \ge 0}$ be a Brownian motion starting at~$x \in (-\eps,\eps)$. As~$T \to \infty$, the probability measures 
$$
    \p_x\left((W_t)_{t\geq 0} \in \cdot \,\, \Bigg\abs\, \sup_{t \in [0,T]} \abs W_t\abs \le \eps\right)
$$ 
converge weakly on~$\CC([0,\infty))$ to the law of a process~$(X_t)_{t \ge 0}$ satisfying the SDE
\begin{equation*}
X_0=x,\qquad \D X_t = \D B_t - \frac{\pi}{2\eps} \tan\left(\frac{\pi X_t}{2\eps}\right) \D t, \quad t \ge 0,
\end{equation*}
where~$(B_t)_{t \ge 0}$ is a standard Brownian motion.
\end{thm}

Interpreting the conditioning in the above theorem as a restriction on the $L_\infty$-norm of a Brownian motion on a long interval, it seems natural to ask for a similar result involving the corresponding $L_2$-norm. This is the subject of the present paper.

Let us set up the required notation: We let~$W=(W_t)_{t \ge 0}$ be a Brownian motion with an arbitrary starting point~$x \in \R$ and define
$$
       I_T:= \int_0^T W_s^2 \D s, \quad T \ge 0.
$$
Our main result reads as follows:

\begin{thm} \label{thmMain} 
Let~$\eps \gne 0$.
As~$T \to \infty$, the probability measures 
$$
   \p_x\left( (W_t)_{t\geq 0} \in \cdot \,\, \abs\, I_T \le \eps T\right)
$$ 
converge weakly on~$\CC([0,\infty))$ to the law of 
an Ornstein--Uhlenbeck process~$(X_t)_{t \ge 0}$ satisfying 
$$
   X_0=x, \qquad \D X_t = \D B_t - \frac{1}{2\eps} X_t \D t, \quad t \ge 0,
$$ 
where~$(B_t)_{t \ge 0}$ is a standard Brownian motion.
\end{thm}

Notice that the random variable~$I_T$ is of order~$T^2$, i.e., $\E_x[I_T] \sim \frac{T^2}{2}$, as $T \to \infty$. Therefore, Theorem~\ref{thmMain} deals with so called small deviation (or small ball) probabilities, as well as Theorem~\ref{thmUni}. In order to show Theorem~\ref{thmMain}, we are going to obtain a new result for the $L_2$ small deviation probabilities of non-centered Brownian motion, cf.~Proposition~\ref{prop:smallBall} below, which may be of interest in its own right.

Interestingly, the limiting processes appearing in both theorems have an interpretation as solutions of a stochastic control problem (energy-efficient approximation), see~\cite{Kar80} and~\cite{LS15}. They also appear in a context related to constructive quantum field theory, cf.~\cite{RS76}.
\bigskip

It is natural to consider other restrictions in the conditioning in Theorem~\ref{thmMain} rather than $I_T\leq \theta T$. In the next theorem, we replace $\theta$ by a positive function $\theta_T$ such that $\theta_T T \ll T^2$ so that the whole small deviations regime is covered. In the setting of non-constant $\theta_T$, one has to use a time- and space-rescaling to obtain a limit result:

\begin{thm} \label{thmMainScaled} 
Let~$\eps:[0,\infty) \to (0,\infty)$ be a function such that $\lim_{T\to\infty} \tfrac{\eps_T}{T}=0$.
As~$T \to \infty$, the probability measures 
$$
   \p_0\left(  \frac{1}{\sqrt{\eps_T}}\, (W_{\eps_T t})_{t \geq 0}  \in \cdot \,\, \bigg\abs\, I_T \le \eps_T T\right)
$$ 
converge weakly on~$\CC([0,\infty))$ to the law of an Ornstein--Uhlenbeck process~$(X_t)_{t \ge 0}$ satisfying 
$$
   X_0=0, \qquad \D X_t = \D B_t - \frac{1}{2} X_t \D t, \quad t \ge 0,
$$ 
where~$(B_t)_{t \ge 0}$ is a standard Brownian motion.
\end{thm}

Theorem~\ref{thmMainScaled} will be deduced from Theorem~\ref{thmMain} by scaling arguments in Section~\ref{sec:Scaling}. We remark that one can also use different starting points (rather than starting point $0$), but a space rescaling then also has to be applied to the starting point, which can be, e.g., $x\sqrt{\theta_T}$ leading to an Ornstein--Uhlenbeck process with the starting point~$x$.

\bigskip

As a first step towards proving Theorem~\ref{thmMain}, we will need an analogous result related to a Brownian motion surviving under a quadratic killing rate, which is of interest in its own right:

\begin{thm}  \label{thmExp}
Given~$\gamma \gne 0$, let~$\eta$ be an exponential random variable with mean~$\frac{2}{\gamma^2}$ independent of~$W$. Then, as~$T \to \infty$, the probability measures $$\p_x((W_t)_{t\geq 0} \in \cdot \,\, \abs \, I_T \le \eta)$$ converge weakly on~$\CC([0,\infty))$ to the law of an Ornstein--Uhlenbeck process~$(X_t)_{t \ge 0}$ satisfying $$X_0=x, \qquad \D X_t = \D B_t - \gamma X_t \D t, \quad t \ge 0,$$ where~$(B_t)_{t \ge 0}$ is a standard Brownian motion.
\end{thm}

The conditioning event~$\{I_T \le \eta\}$ can be interpreted as the survival until at least time~$T$ of a Brownian motion killed with the rate~$\gamma^2 z^2/2$ when passing through a point~$z$.
\newpage

A  result equivalent to Theorem~\ref{thmExp} can be found 
in Example~4.3 of~\cite{RVY06} in the context of Brownian motion perturbed by
rather general exponential weights. It is proved there by an analytic semi-group approach. Here we suggest another, completely different proof, which is rather straightforward and highlights the role of the Markov property and small deviation probabilities in a very transparent way. We also include the proof here to make this paper more self-contained and, finally, because we believe that the present proof is so simple that it can be used for teaching purposes.

The proof of our main result, Theorem~\ref{thmMain}, is not so straightforward. While our proof of Theorem~\ref{thmExp} uses the classical Doob h-transform-type approach (even though no h-transform appears explicitly), Theorem~\ref{thmMain} cannot be proved in this way. This is due to the fact that the condition~$I_T \leq \eps T$ cannot be checked ``on the fly'' by the process. This is different for the condition $\sup_{t\in[0,T]} |W_t|\leq \eps$ (in Theorem~\ref{thmUni}) or the conditioning on not having been killed (in Theorem~\ref{thmExp}), which are more compatible with the Markov property. Our approach to proving Theorem~\ref{thmMain} is to reduce it to Theorem~\ref{thmExp}. The key will be to represent Brownian motion conditioned on not having been killed (as considered in Theorem~\ref{thmExp}) as a mixture (in~$y$) of Brownian motion conditioned on~$I_T\leq y$ (cf.~representation~\eqref{eqn:integralrepofkilledprocdistr} below), then identify the $y$-area with most weight (which will turn out to be of order~$\eps T$), and finally disintegrate in a suitable way.
\medskip

The outline of this paper is as follows. In Section~\ref{sec:exp}, we give a self-contained proof of Theorem~\ref{thmExp}. Then we prepare the proof of our main result by studying the (non-centered) small deviation probabilities of Brownian motion in $L_2$-norm in Section~\ref{sec:smallballs}, which is of interest in its own right. The proofs of Theorems~\ref{thmMain} and~\ref{thmMainScaled} are given in Sections~\ref{sec:connection} and~\ref{sec:Scaling}, respectively. We conclude in Section~\ref{sec:Q} with some remarks on possible generalizations.

\section{Proof of Theorem~\ref{thmExp}} \label{sec:exp}
We follow the classical approach of proving convergence of finite-dimensional distributions and tightness. From formula~1.1.9.3 in~\cite{BS96}, we get
\begin{eqnarray}
\label{BS96}
\p_x(I_T \le \eta)
&=&\E_x \exp\left( - \frac{\gamma^2}{2} \int_0^T W_s^2 \D s \right) \notag\\
&=& \frac{1}{\sqrt{\cosh(T \gamma)}} \, \exp\left(- \frac{x^2 \gamma \sinh(T \gamma)}{2 \cosh(T\gamma)}\right), \quad T \gne 0,\, x \in \R.
\end{eqnarray}
This implies
\begin{equation}
\label{eqn:tauComp}
\p_x(I_T \le \eta) \le \p_0(I_T \le \eta), \quad x \in \R,
\end{equation}
as well as
\begin{eqnarray}
 \lim_{T\to\infty} \frac{\p_y(I_{T-t} \le \eta)}{\p_x(I_T \le \eta)}
&= &\lim_{T\to\infty} \sqrt{\frac{\cosh(T\gamma)}{\cosh((T-t)\gamma)}} \, 
    \exp\left(- \frac{y^2 \gamma  \sinh((T-t) \gamma)}{2\cosh((T-t)\gamma )}
     + \frac{x^2 \gamma \sinh(T\gamma)}{2\cosh(T\gamma )}\right) \notag
\\
&= &\lim_{T\to\infty} \sqrt{\frac{\e^{T\gamma }}{\e^{(T-t)\gamma }}} \, 
     \frac{\exp\left(- \frac{y^2\gamma}{2}\right)}{\exp\left(-\frac{x^2\gamma}{2}\right)} \notag
\\
&= & \e^{t \frac{\gamma}{2}} \, 
   \frac{\exp\left(- \frac{y^2\gamma}{2}\right)}{\exp\left(-\frac{x^2\gamma}{2}\right)}, \qquad x,y \in \R.
\label{eqn:quotient}
\end{eqnarray}
Now let $$p_t^\gamma(x,y):=\p_x\left( W_t \in \D y, I_t \le \eta\right),\quad x,y \in \R,$$ be the (non-probability) transition density of~$W$ killed when the process~$(I_t)_{t \ge 0}$ hits~$\eta$. Further, let~$t_1,\dots,t_d \gne 0$ with~$t_1 \lne \dots \lne t_d$ and~$C_1,\dots,C_d \in \BB(\R)$ as well as~$y_0:=x$ and~$t_0:=0$.  Using the Markov property together with the memorylessness of the exponential distribution in the first step as well as the dominated convergence theorem, which is applicable due to~\eqref{eqn:tauComp}, in the second step, we obtain
\begin{eqnarray*}
   &&\lim_{T \to \infty} \p_x(B_{t_1} \in C_1, \dots, B_{t_d} \in C_d\, |\, I_T \le \eta)
\\
   &=&  \lim_{T \to \infty} \int_{\R^d} \1_{C_1\times \dots \times C_d}(\boldsymbol{y}) \prod_{i=1}^d p_{t_i-t_{i-1}}^\gamma (y_{i-1},y_i) \cdot \frac{\p_{y_d}( I_{T-t_d} \le \eta)}{\p_x(I_T \le \eta)}\, \D \boldsymbol{y}
\\
&=&  \int_{\R^d} \1_{C_1\times \dots \times C_d}(\boldsymbol{y}) \prod_{i=1}^d p_{t_i-t_{i-1}}^\gamma(y_{i-1},y_i) \cdot    \e^{t_d \frac{\gamma}{2}} \frac{\exp\left(-\frac{y_d^2\gamma}{2} \right)}{\exp\left(- \frac{x^2\gamma}{2}\right)} \D \boldsymbol{y}
\\
&=& \int_{\R^d} \1_{C_1\times \dots \times C_d}(\boldsymbol{y}) \prod_{i=1}^d \left(p_{t_i-t_{i-1}}^\gamma(y_{i-1},y_i) \e^{(t_i-t_{i-1})\frac{\gamma}{2}}   \frac{\exp\left(-\frac{y_i^2\gamma}{2}\right)}{\exp\left(-\frac{y_{i-1}^2\gamma}{2}\right)}\, \right) \D \boldsymbol{y},
\end{eqnarray*}
with~$\boldsymbol{y}=(y_1,\dots,y_d)$. The convergence of the finite-dimensional distributions to the claimed limit follows from the fact that 
$$
      \tilde p_{t}(x,y):= p_t^\gamma (x,y) \cdot  
          \e^{t \frac{\gamma}{2}} \
          \frac{\exp(- \frac{y^2\gamma}{2})}{\exp(-\frac{x^2\gamma}{2})}, \quad x,y \in \R,
$$
is the transition density of an Ornstein--Uhlenbeck process~$X$ defined as a solution 
of the SDE
$$
    \D X_t = \D B_t - \gamma X_t \, \D t. 
$$
For the sake of completeness, we prove this fact in what follows: It is well known that~$X$ is a Gaussian Markov process with normal transition distributions. Namely, for all~$x,t_0,t\in \R$, conditioned on~$\{X_{t_0}=x\}$,  the random variable~$X_{t_0+t}$ is a Gaussian with mean~$\e^{-\gamma t} x$ and variance~$(2\gamma)^{-1} ( 1 - \e^{-2\gamma t})$. On the other hand, by formula~1.1.9.7 in~\cite{BS96}, we have
\begin{eqnarray*}
    && p_t^\gamma(x,y) \cdot  \e^{t\frac{\gamma}{2} } \frac{\exp\left(- \frac{y^2\gamma}{2}\right)}{\exp\left(-\frac{x^2\gamma}{2}\right)}\, \D y
\\
     &=&\E_x\left[ \e^{-\frac{\gamma^2}{2}\int_0^t W_s^2 \D s} \1_{ W_t \in \D y} \right]  
        \cdot  \e^{t\frac{\gamma}{2} } \frac{\exp\left(- \frac{y^2\gamma}{2}\right)}{\exp\left(-\frac{x^2\gamma}{2}\right)}\, \D y
\\
      &=& \frac{\sqrt{\gamma}}{\sqrt{2\pi \sinh(t\gamma)}} \exp\left( - \gamma \frac{(x^2+y^2)\cosh(t\gamma)-2xy}{2 \sinh(t\gamma)} \right)  
          \cdot   \e^{t\frac{\gamma}{2} } \frac{\exp\left(- \frac{y^2\gamma}{2}\right)}{\exp\left(-\frac{x^2\gamma}{2}\right)}\, \D y
\\
      &=&\frac{\sqrt{\gamma}}{\sqrt{2\pi \sinh(t\gamma)}} \,
      \e^{t \frac{\gamma}{2} } \cdot
\\
&& \cdot \exp\left[ - \frac{\gamma}{2}\left(   x^2 \left( \frac{\cosh(t \gamma)}{\sinh(t\gamma)} -1\right) - x y \frac{2}{\sinh(t\gamma)}+y^2 \left(\frac{\cosh(t \gamma)}{\sinh(t\gamma)} +1\right) \right)\right] \, \D y
\\
&=&\frac{\sqrt{2\gamma}}{\sqrt{2\pi (1- \e^{-2\gamma t})}} \cdot
\\
&& \cdot \exp\left[ - \frac{\gamma}{2}\left(  (x \e^{-\gamma t})^2 \frac{2}{1-\e^{-2\gamma t}} - 2  x\e^{-\gamma t} \, y\, \frac{2}{1-\e^{-2\gamma t}}+y^2 \frac{2}{1-\e^{-2\gamma t}}\right)\right] \, \D y
\\
&=&\frac{\sqrt{2\gamma}}{\sqrt{2\pi (1- \e^{-2\gamma t})}} \cdot \exp\left[ - \frac{1}{2} \, (y-x \e^{-\gamma t})^2 \frac{2\gamma}{1-\e^{-2\gamma t}} \right] \, \D y, \qquad x,y \in \R,
\end{eqnarray*}
which is exactly the density of the normal distribution with  required parameters.
\newpage

To prove that the family $(\p_x\left( (W_t)_{t\geq 0} \in \cdot\,\,\abs\, I_T \le \eta)\right)_{T\ge 0}$ 
is tight, let~$t_0 \gne 0$ and~$t_1,t_2 \in [0,t_0]$ with~$t_1 \le t_2$. Using the Markov property and~\eqref{eqn:tauComp}, we get
\begin{eqnarray*}
   \E_x \left[\abs W_{t_2}-W_{t_1}\abs^4\,\big\abs\, I_T \le \eta\right]
   &\le& \frac{1}{\p_x(I_T \le \eta)} \E_x\left[\abs W_{t_2}-W_{t_1}\abs^4\1_{\{I_{T}-I_{t_2} \le \eta\}} \right]
\\
   &=&  \frac{1}{\p_x(I_T \le \eta)} \E_x\left[\abs W_{t_2}-W_{t_1}\abs^4 \p_{W_{t_2}}(I_{T-t_2} \le \eta) \right]
\\
   &\le& \frac{\p_0(I_{T-t_2} \le \eta)}{\p_x(I_T \le \eta)}\E_x\abs W_{t_2}-W_{t_1}\abs^4
\\
   &\le& \frac{\p_0(I_{T-t_0} \le \eta)}{\p_x(I_T \le \eta)} 3 \abs t_2-t_1\abs^2, \quad T \ge t_0.
\end{eqnarray*}
In view of~\eqref{eqn:quotient} and the continuity of the probabilities in~$T$, we deduce that there exists a constant~$C \gne 0$ with 
$$ 
  \E_x \left[\abs W_{t_2}-W_{t_1}\abs^4\,\big\abs\, I_T \le \eta\right] \le C \abs t_2-t_1\abs^2, \quad t_1,t_2 \in [0,t_0],\ T \ge 0.
$$

Applying a standard tightness criterion (see, e.g., Theorem~12.3 in~\cite{B68}), we see that
the family of measures under consideration is tight on~$\CC([0,t_0])$. Corollary~5 in~\cite{Whi70} yields tightness on~$\CC([0,\infty))$. 
\hfill \qed

\section{Small deviation probabilities of $I_T$} 
\label{sec:smallballs}
This section is devoted to the analysis of the small deviation probabilities of~$I_T$, i.e., of $$
P_{x,T}(y):= \p_x(I_T\leq y),$$ for~$T \to \infty$, $y\sim \theta T$, and moderate and small~$x$. In the first subsection, we recall a result of Li and Linde, \cite{LL93}, which deals with non-centered small deviation probabilities in Hilbert norm of a general Gaussian random vector. Then we apply this result to our special situation in Section~\ref{sec:LLapplication}. Section~\ref{sec:smallballuniform} then gives the next steps in our analysis: A uniform version of the small deviation result for varying starting points~$x$ and $y$-values as well as a result on ratios of~$P_{x,T}(y)$ for different starting points, different time frames, and different amounts of~$y$.

\subsection{Non-centered small deviation probabilities}
Let~$X$ be a Gaussian random vector with values in a Hilbert space. Assume that one can write
$$
    X=\sum_{j=1}^\infty \lambda_j^{1/2} \xi_j e_j,
$$
where~$(e_j)$ is an orthonormal basis of the Hilbert space, $\lambda_1\geq \lambda_2 \geq \ldots > 0$ with $\sum_{j=1}^\infty \lambda_j < \infty$, and~$(\xi_j)$ is an i.i.d.~sequence of standard normal random variables. Further, consider an element $a=\sum_{j=1}^\infty \alpha_j e_j$ of the Hilbert space. Then define the functions
$$
    \psi(\gamma):=\sum_{j=1}^\infty \frac{\alpha_j^2 \gamma}{1+2\lambda_j \gamma},\qquad\text{and}\qquad \chi(\gamma):=\frac{1}{2} \sum_{j=1}^\infty \log\left( 1 + 2 \lambda_j \gamma\right), \qquad \gamma \ge 0.
$$
Let $f,R : [0,\infty) \to (0,\infty)$ be two functions such that $R(T)\to 0$ for $T\to\infty$.
Then the equation
\begin{equation} \label{eqn:asymptoticeqn}
    R(T) = f(T) \psi'(\gamma) + \chi'(\gamma)
\end{equation}
has a unique solution $\gamma=\gamma(T)$ for large $T$ (cf.~Remark~3.1 in~\cite{LL93}). 
Further, set
$$
    \beta(T):=-f(T)\gamma(T)^2\psi''(\gamma(T))-\gamma(T)^2\chi''(\gamma(T)).
$$

\begin{lem} \label{lem:lilinde} Under the above assumptions, when $T\to\infty$,
$$
    \p( ||X-f(T)^{1/2} a||^2 < R(T) ) \sim \frac{1}{\sqrt{2\pi \beta(T)}} \, \exp\left( \gamma(T) R(T) - f(T) \psi(\gamma(T)) - \chi(\gamma(T)) \right).
$$
\end{lem}

This result can be found as Corollary~3.3 in~\cite{LL93}. Some similarly applicable results can be found in~\cite{L97}.

\subsection{Application to $I_T$} \label{sec:LLapplication}
We apply Lemma~\ref{lem:lilinde} to the Hilbert space $L_2[0,1]$, to the process $X:=\bar W$, where $\bar W$ is a standard Brownian motion, $f(T)=x^2 T^{-1}$ for moderate $x$, $a\equiv 1$, and $R(T)=yT^{-2}$ for moderate~$y$. This will give us the following small deviation result for $I_T$.

\begin{prop}
\label{prop:smallBall} ~
Let~$x=x(T)$ be a function with~$x^2\leq T^{1/4}$. Let~$y=y(T)$ be a positive function with~$T^2 y^{-1}\to \infty$ as~$T \to \infty$. Then we have
\begin{equation}
\label{eqn:smallballasymptotics}
P_{x,T}(y) = \p_x\left(I_T \leq  y\right) \sim  \frac{4}{\sqrt{\pi T^2 y^{-1}}}\, \exp\left(  -\frac{(T+x^2)^2}{8y} \right), \quad T \to \infty.
\end{equation}
\end{prop}

The asympototics~\eqref{eqn:smallballasymptotics} in the centered case, i.e., for~$x=0$, is classical. It can be found, e.g., in~\cite[p.~511]{GHLT04}, or in~\cite{NP23}. 

While the above proposition is the classic formulation for small deviation probabilities, we will need a more precise, uniform (in the variables~$x$ and~$y$) bound in Proposition~\ref{prop:smallBall2} below, which is in fact equivalent to Proposition~\ref{prop:smallBall}.

\begin{proof} 
Let us describe how to change the question to make it fit to the setup of Lemma~\ref{lem:lilinde}.
We write~$W=x+\bar W$ with a standard Brownian motion~$\bar{W}$, which, according to the Karhunen-Lo\`eve representation, has a representation in~$L_2[0,1]$ given by
$$
     \bar W_t = \sum_{j=1}^\infty \sqrt{\lambda_j} \xi_j e_j(t),\qquad t\in[0,1],
$$
with independent and standard normally distributed~$\xi_1,\xi_2,\ldots$ and
$$
    \lambda_j := \frac{1}{(j-1/2)^2\pi^2},\quad  e_j(t):=\sqrt{2} \sin( (j-1/2) \pi t), \qquad j \in \N.
$$
We recall that~$(e_j)_{j \in \N}$ is an orthonormal basis in~$L_2[0,1]$. Further, one computes $$\alpha_j:=\langle 1 , e_j\rangle = \sqrt{2} (j-1/2)^{-1} \pi^{-1}, \quad j \in \N.$$
This gives
$$
a:= 1 = \sum_{j=1}^\infty \langle 1, e_j\rangle \cdot e_j(t)  =  \sum_{j=1}^\infty \frac{\sqrt{2} }{(j-1/2)\pi}\cdot e_j(t) = \sum_{j=1}^\infty \alpha_j\cdot e_j(t), \quad t \in [0,1].
$$
We note that~$\alpha_j^2 = 2 \lambda_j$ for each~$j \in \N$. A simple substitution, Brownian scaling (recall that~$\bar W$ is a standard Brownian motion), and the fact that~$|\bar W+x|^2$ and~$|\bar W - x|^2$ have the same distribution imply that
\begin{eqnarray}
\p_x\left(  I_T \leq y\right) &=& \p_x\left( \int_0^1 W_{Ts}^2 \D s \leq  y T^{-1} \right)
\notag  \\
&=& \p_0\left(  \int_0^1 |T^{1/2} \bar W_s + x|^2 \D s \leq  y T^{-1}\right)
\notag \\
& =& \p_0\left(  \int_0^1 | \bar W_s - |x| T^{-1/2}|^2 \D s \leq  y T^{-2}\right)
\notag \\
& =& \p_0\left(  \int_0^1 | \bar W_s - f(T)^{1/2} a|^2 \D s \leq R(T)\right), \label{eqn:rescalingll}
\end{eqnarray}
where~$R(T):= y T^{-2}$, $f(T):=x^2 T^{-1}$, and~$a:=1$.
In the context of Lemma~\ref{lem:lilinde}, the functions~$\psi$ and~$\chi$ satisfy~$\psi(0)=0$ and
$$
\psi'(\gamma)= \sum_{j=1}^\infty \frac{\alpha_j^2}{(1+2\lambda_j\gamma)^2}
= \sum_{j=1}^\infty\frac{2(j-1/2)^2 \pi^2}{((j-1/2)^2 \pi^2+2\gamma )^2}
= \frac{\tanh(\sqrt{2\gamma})}{2\sqrt{2\gamma}}  +  \frac{1}{2}  \cdot  \op{sech}(\sqrt{2\gamma})^2, \quad \gamma > 0,
$$
as well as~$\chi(0)=0$ and
$$
\chi'(\gamma) = \sum_{j=1}^\infty   \frac{\lambda_j}{1+2\lambda_j\gamma}  
= \sum_{j=1}^\infty  \frac{1}{(j-1/2)^2 \pi^2 +2\gamma} 
= \frac{\tanh(\sqrt{2\gamma})}{2\sqrt{2\gamma}}, \quad \gamma > 0.
$$
We now analyse what happens for~$\gamma\to\infty$. Recall that
$$
\tanh(z)=\frac{\e^z - \e^{-z}}{\e^z+\e^{-z}} = 1 +O(\e^{-2z})
\qquad
\text{and}
\qquad
\op{sech}(z)=\frac{2 \e^{z}}{\e^{2z}+1} =O( \e^{-z}),
$$
as $z\to\infty$. Consequently,
$$
\psi'(\gamma) = \frac{1}{2 \sqrt{2\gamma}} ( 1 + O(\sqrt{\gamma} \e^{-2\sqrt{2\gamma}})) \quad\text{and}\quad \chi'(\gamma)=\frac{1}{2 \sqrt{2\gamma}} (1 +O(\e^{-2\sqrt{2\gamma}})), \qquad  \text{as $\gamma\to\infty$}.
$$
Now let $\gamma=\gamma(T)$ be the unique solution of~\eqref{eqn:asymptoticeqn}, i.e., of
\begin{equation}
\label{*}
y T^{-2} = x^2 T^{-1} \psi'(\gamma) + \chi'(\gamma),
\end{equation}
and define $\ell=\ell(T)$ by
\begin{equation*}
\gamma= \frac{y^{-2} T^4}{8} (1 +T^{-1} x^2)^2 ( 1 + \ell).
\end{equation*}
Plugging this into~\eqref{*} and noting that the general theory of~\cite{LL93} implies $\lim_{T \to \infty} \gamma(T) = \infty$, we obtain
\begin{equation}
\label{**}
\sqrt{1+\ell} = 1+O(\sqrt{\gamma} \e^{-2\sqrt{2\gamma}})
\end{equation}
proving $\lim_{T \to \infty} \ell(T) = 0$. Noting $\sqrt{1+z} = 1 + z/2 + o(z)$ as~$z \to 0$ and using~\eqref{**} again, we deduce $$\ell = O(\sqrt{\gamma} \e^{-2\sqrt{2\gamma}}) = o( \e^{-\sqrt{2\gamma}}) = o( \e^{-y^{-1}T^2/2 })$$ proving
$$
\gamma = \gamma(T) = \frac{y^{-2} T^4}{8} (1 +T^{-1} x^2)^2 ( 1 + o(\e^{-y^{-1} T^2/2})).
$$
This choice of~$\gamma$ yields
$$
\psi(\gamma)= \sum_{j=1}^\infty \frac{\alpha_j^2 \gamma}{1+2\lambda_j\gamma}
= 2 \gamma \cdot  \frac{\tanh(\sqrt{2\gamma})}{2\sqrt{2\gamma}} 
\\
=
\frac{y^{-1} T^2}{4} + \frac{y^{-1}Tx^2}{4} + o(1)
$$
and
$$
\chi(\gamma)
=
\int_0^\gamma \frac{\tanh(\sqrt{2z})}{2 \sqrt{2z}} \D z 
= \frac{1}{2} \log \cosh( \sqrt{2\gamma} )
= \frac{y^{-1}T^2}{4}(1+T^{-1} x^2) - \log \sqrt{2} + o(1).
$$
Differentiation also gives (using that~$x^2 \leq T^{1/4}$ implies~$f(T)\leq T^{-3/4}$ uniformly)
$$
\beta(T,\gamma):=-f(T) \gamma^2 \psi''(\gamma) - \gamma^2 \chi''(\gamma) \sim  - \gamma^2 \, \frac{(-1/2) }{2 \sqrt{2} \gamma^{3/2}}  \sim \frac{y^{-1} T^2}{16}.
$$
Further,
$$
\gamma R(T) = \frac{y^{-1} T^2}{8} (1+T^{-1} x^2)^2 + o(1).
$$
Using Lemma~\ref{lem:lilinde},
we finally obtain from the representation~\eqref{eqn:rescalingll} that
\begin{eqnarray*}
&& 
\p_x( I_T \leq y)
\\
&\sim&
 \frac{1}{\sqrt{2\pi \beta(T,\gamma)}}\, \exp( \gamma R(T) - f(T) \psi(\gamma)-\chi(\gamma))
 \\
&\sim & \frac{4}{\sqrt{2\pi T^2 y^{-1}}}\, \exp\left( \frac{y^{-1}T^2}{8} + \frac{y^{-1} T x^2}{4} + \frac{y^{-1} x^4}{8} - \frac{y^{-1}T x^2}{4}-\frac{y^{-1} x^4}{4} -\frac{y^{-1}T^2}{4}-\frac{y^{-1} T x^2}{4}+\log\sqrt{2}\right)
\\
& =& \frac{4}{\sqrt{\pi T^2 y^{-1}}}\, \exp\left( -\frac{(T+x^2)^2}{8y}\right),
\end{eqnarray*}
as claimed.
\end{proof}

\subsection{Asymptotic analysis of $\p_x(I_T \le y)$ and ratios thereof} \label{sec:smallballuniform}

The next step is to make Proposition~\ref{prop:smallBall} ``uniform'' because we need the estimate simultaneously in a range of possible $x$ and $y$.

\begin{prop}
\label{prop:smallBall2}
For any~$\delta\in (0,1)$, there is a~$T^0_{\delta}>0$ such that, for any~$T\geq T^0_{\delta}$, any~$x\in\R$ with~$x^2\leq T^{1/4}$, and any~$y>0$ with~$T^2 y^{-1}\geq T^0_{\delta}$, we have
\begin{equation}
\label{eqn:smallballasymptoticsprime}
(1-\delta) \frac{4}{\sqrt{\pi T^2 y^{-1}}}\, \exp\left(  -\frac{(T+x^2)^2}{8y} \right)  \leq 
P_{x,T}(y)
\leq (1+\delta) \frac{4}{\sqrt{\pi T^2 y^{-1}}}\, \exp\left(  -\frac{(T+x^2)^2}{8y} \right).
\end{equation}
\end{prop}

As a technical side remark, we note that the condition~$x^2\leq T^{1/4}$ in Propositions~\ref{prop:smallBall} and~\ref{prop:smallBall2} can be weakened up to~$x^2 T^{-1}=o(1)$, but the $o(1)$-term is then connected to~$T^0_\delta$. We do not pursue this in order to make the situation less complicated.

\begin{proof}
We just use a contradiction argument to derive the statement from Proposition~\ref{prop:smallBall}: Assume that there is a strictly increasing sequence~$T_n\to\infty$ as well as sequences~$y_n>0$ with~$T_n^2 y_n^{-1}\to \infty$ and~$x_n\in\R$ with~$x_n^2 \leq T_n^{1/4}$ such that, e.g., the upper bound in~\eqref{eqn:smallballasymptoticsprime} is wrong along these subsequences. Then consider functions~$y(T)>0$ defined via~$y(T):=y_n$ for~$T\in[T_n,T_{n+1})$ and~$x(T)$ defined via~$x(T):=x_n$ for~$T\in[T_n,T_{n+1})$. For these functions, the upper bound in~\eqref{eqn:smallballasymptotics} would be wrong, while $T^2 y(T)^{-1} \geq T_n^2 y_n^{-1}\to \infty$ and $x(T)^2=x_n^2\leq T_n^{1/4}\leq T^{1/4}$. Similarly, one can proceed with the lower bound in~\eqref{eqn:smallballasymptoticsprime}, which is derived from the lower bound in~\eqref{eqn:smallballasymptotics}.
\end{proof}

The following corollary is a consequence of Proposition~\ref{prop:smallBall2} for ratios of the~$P_{x,T}(y)$ with different starting points~$x$, different time frames~$T$, and different amounts of allowed $L_2$-norm~$y$, where the latter is of order~$\eps T$.
While part~(a) is a simple consequence of Proposition~\ref{prop:smallBall}, the more precise (``uniform'') description in~(b) follows from Proposition~\ref{prop:smallBall2} and will be required in the proof of Theorem~\ref{thmMain}.

\begin{cor}
\label{l:long_limit}
Fix~$t_0>0$ and~$x \in \R$.
\begin{itemize}
    \item[(a)] Let~$y=y(T)$ be a positive function such that~$y\sim\eps T$. For any fixed~$z\in\R$ and~$y_0\geq 0$, we have
\begin{equation}
\label{conj}
g_x(t_0,z,y_0) := \lim_{T\to\infty} \frac{P_{z,T-t_0}(y-y_0)}{P_{x,T}(y)} = \exp\left( \frac{t_0}{4\eps}-\frac{z^2}{4\eps}+\frac{x^2}{4\eps} -\frac{y_0}{8\eps^2} \right).
\end{equation}
\item[(b)]
For any~$\delta\in (0,1)$ and any~$M>0$, there is a~$T^1_{\delta,M}>0$ such that, for any~$T\geq T^1_{\delta,M}$, any~$y$ satisfying $\eps T -M\sqrt{T}\leq y \leq \eps T+M\sqrt{T}$, any~$z\in\R$ with~$z^2\leq T^{1/4}$, and any~$y_0 \ge 0$ with~$y_0^2\leq T^{1/2}$, we have
\begin{equation}
\label{conjprime}
(1-\delta) g_x(t_0,z,y_0)\leq \frac{P_{z,T-t_0}(y-y_0)}{P_{x,T}(y)} \leq (1+\delta) g_x(t_0,z,y_0).
\end{equation}
\end{itemize}
\end{cor}

\begin{proof} Let us start with~(a). We use Proposition~\ref{prop:smallBall} for fixed~$x,z$ and~$y\sim \theta T$ and compute the resulting asymptotics for the fraction~$P_{z,T-t_0}(y-y_0)/P_{x,T}(y)$:
 \begin{eqnarray}
   F  &:=& \frac{\frac{4}{\sqrt{\pi (T-t_0)^2 (y-y_0)^{-1}}}\, \exp\left(  - 
        \frac{(T-t_0+z^2)^2}{8(y-y_0)} \right)}{\frac{4}{\sqrt{\pi T^2 y^{-1}}}\, \exp\left(  -\frac{(T+x^2)^2}{8y} \right)}\notag
 \\
    &=&
     \frac{\frac{1}{\sqrt{T^2 (1-t_0/T)^2 y^{-1} (1-y_0/y)^{-1}}}\, \exp\left(  -\frac{(T-t_0)^2+2(T-t_0) z^2+z^4}{8y(1-y_0/y)} \right)}{\frac{1}{\sqrt{T^2 y^{-1}}}\, \exp\left(  -\frac{T^2+2Tx^2+x^4}{8y} \right)} \notag
 \\
    &\sim&
 \frac{\exp\left(  -\frac{T^2-2Tt_0+t_0^2+2Tz^2-2t_0z^2+z^4}{8y}(1+\frac{y_0}{y}) \right)}{\exp\left(  -\frac{T^2+2Tx^2+x^4}{8y} \right)} \notag
 \\
 &\sim&
 \frac{\exp\left(  -\frac{T^2}{8y}+\frac{2Tt_0}{8y}-\frac{t_0^2}{8y}-\frac{2Tz^2}{8y}+\frac{2t_0z^2}{8y}-\frac{z^4}{8y}- \frac{T^2}{8y} \frac{y_0}{y} \right)}{\exp\left(  -\frac{T^2}{8y}-\frac{2Tx^2}{8y}-\frac{x^4}{8y} \right)} \notag
 \\
 &\sim&
 \frac{\exp\left( \frac{t_0}{4\eps}-\frac{z^2}{4\eps} - \frac{y_0}{8\eps^2} \right)}{\exp\left( -\frac{x^2}{4\eps} \right)} \notag\\
 &=& g_x(t_0,z,y_0). \label{eqn:quotientcomputation}
 \end{eqnarray}
This shows part~(a). To prove~(b) fix $\delta\in (0,1)$ and~$M>0$. On the one hand, note that the computation in~\eqref{eqn:quotientcomputation} can be turned into a uniform estimate: There exists a suitable~$\tilde T^1_{\delta,M}>0$ such that
$$
\left(1-\frac{\delta}{5}\right) g_x(t_0,z,y_0) \leq F \leq\left(1+\frac{\delta}{5}\right)  g_x(t_0,z,y_0),
$$
for~$T\geq \tilde T^1_{\delta,M}$, $y \in [\eps T -M\sqrt{T}, \eps T+M\sqrt{T}]$, $z^2\leq T^{1/4}$, and~$y_0^2\leq T^{1/2}$.
On the other hand, applying Proposition~\ref{prop:smallBall2} with $\delta/5 \in (0,1)$ guarantees the existence of a constant $T^0_{\delta/5} \gne 0$ such that $$1-\frac{\delta}{5} \le P_{z,T-t_0}(y-y_0) \le 1+\frac{\delta}{5} \qquad \text{and} \quad 1-\frac{\delta}{5} \le P_{x,T}(y) \le 1+\frac{\delta}{5}$$ are satisfied for~$T \ge T^0_{\delta/5}$ and for all~$y$, $z$ and~$y_0$ in the mentioned areas. Recalling~$\delta \in(0,1)$, we deduce $$\frac{P_{z,T-t_0}(y-y_0)}{P_{x,T}(y)} \le \frac{1+ \delta/5}{1- \delta/5} F \le \frac{(1+ \delta/5)^2}{1- \delta/5} g_x(t_0,z,y_0) \le (1+\delta) g_x(t_0,z,y_0)$$ for~$T\geq \max(\tilde T^1_{\delta,M},T^0_{\delta/5})=:T^1_{\delta,M}$ and for all~$y$, $z$ and~$y_0$ in the mentioned areas. Similarly, the lower bound of~\eqref{conjprime} is verified.
\end{proof}

Similar to the comment after Proposition~\ref{prop:smallBall2}, one can weaken the assumptions of part~(b) in Corollary~\ref{l:long_limit} to~$z^2T^{-1/2}=o(1)$ as well as~$y_0^2 T^{-1}=o(1)$. However, then~$T^1_{\delta,M}$ will depend on the $o(1)$-terms, which becomes more complicated than necessary for our purposes.

\section{Connection between the two conditioning problems}\label{sec:connection}
In this section, we finally prove our main result, Theorem~\ref{thmMain}, by reducing it to Theorem~\ref{thmExp}.

First, we need the following technical notation: For fixed~$x\in\R$, $t_0>0$, and~$m\in\N$, we set
\begin{equation} \label{eqn:defgm}
    \GG_m := \{ W_{t_0}^2 \leq m^{1/4},I_{t_0}^2 \leq m^{1/2} \} \in \sigma(W_s, s \in [0,t_0]),
\end{equation}
where~$(W_t)_{t\geq 0}$ is our Brownian motion started at~$x$.
The background of this notation is the following. For~$T\geq m$ and on a set~$D\subseteq\GG_m$, we have~$W_{t_0}^2 \leq T^{1/4}$ and~$I_{t_0}^2 \leq T^{1/2}$ so that the estimate in~\eqref{conjprime} can be applied to the fraction
$$
    \frac{P_{W_{t_0},T-t_0}(y-I_{t_0})}{P_{x,T}(y)}
$$
as long as $\theta T - M\sqrt{T}\leq y \leq \theta T + M \sqrt{T}$ and~$T$ is sufficiently large.

The next lemma extends the bounds for the fraction treated in~\eqref{conj} and~\eqref{conjprime} to the ratio of conditional probabilities, see~\eqref{all_y} below. Again part~(a) is a simple version that serves for illustration, while part~(b) is a ``uniform'' version that we shall use in the proof of Theorem~\ref{thmMain}. 

\begin{lem} \label{lem:Dconnect}
Fix~$x\in\R$ and~$t_0>0$.
\begin{itemize}
\item[(a)]    
    Let~$y=y(T)$ be a positive function such that~$y\sim\eps T$ and let $D\in\sigma(W_s,s \in [0,t_0])$ with~$\p_x(D)\gne 0$. Then
\begin{equation} \label{all_y}
       \lim _{T\to\infty}  \frac{  \p_x(D|I_T\le y)}{ \p_x(D|I_T\le \eps T) }=1.
\end{equation}
\item[(b)] Fix~$\delta\in (0,1)$, $m\in\N$, $M>0$ and define~$\GG_m$ as in~\eqref{eqn:defgm}. Let~$D\subseteq\GG_m$ be such that $D\in\sigma(W_s,s \in [0,t_0])$ and~$\p_x(D)\gne 0$. Then there exists a~$T^2_{\delta,M}>0$ such that for all~$y$ with $\eps T -M\sqrt{T}\leq y \leq \eps T+M\sqrt{T}$ and all~$T\geq \max(m,T^2_{\delta,M})$ we have
\begin{equation} \label{all_yprime}
1-\delta \leq \frac{\p_x(D|I_T\le y)}{ \p_x(D|I_T\le \eps T) }\leq 1+\delta.
\end{equation}
\end{itemize}
\end{lem}

\begin{proof} 
First, by the Markov property, we have
\begin{equation} \label{eqn:markovdr}
   \p_x(D\cap\{I_T\le y\}) = \E_x \left[\1_{D} \, P_{W_{t_0},T-t_0}(y-I_{t_0}) \right].
\end{equation}
We start with the proof of part~(a). We start with the proof of part~(a) and observe
$$
P_{z,T-t_0}(y-y_0) \le P_{z,T-t_0}(y) \le P_{0,T-t_0}(y), \quad z \in \R,\ y_0 \ge 0,
$$
where the second inequality follows from Anderson's inequality or, alternatively, may be shown by a coupling argument.
Consequently, we can combine~\eqref{eqn:markovdr} and~\eqref{conj} with the dominated convergence theorem to get
\begin{equation*}
    \lim_{T\to\infty}  \p_x(D|I_T\le y)
    = \lim_{T\to\infty} \E_x \left[\1_D \,\frac{P_{W_{t_0},T-t_0}(y-I_{t_0})}{P_{x,T}(y)}\right]
    =   \E_x \left[\1_D \, g_x(t_0,W_{t_0},I_{t_0})  \right].
 \end{equation*}
 It can be checked with the help of the explicit formula in~\eqref{conj} that the latter is a positive, finite constant for any~$t_0>0$ and any~$D\in\sigma(W_s,s \in [0,t_0])$ with~$\p_x(D)\gne 0$. This yields~\eqref{all_y}.

To see part~(b), we use the uniform version~\eqref{conjprime} instead of~\eqref{conj}. Equation~\eqref{eqn:markovdr} and the assumption that~$D\subseteq\GG_m$ entail
\begin{eqnarray} \label{eqn:whatestimate}
   \p_x(D|I_T\le y)
   = \E_x \left[\1_{D} \, \frac{P_{W_{t_0},T-t_0}(y-I_{t_0})}{P_{x,T}(y)} \right] = \E_x \left[\1_{D\cap \GG_m} \, \frac{P_{W_{t_0},T-t_0}(y-I_{t_0})}{P_{x,T}(y)} \right].
\end{eqnarray}
Now, on~$D\cap \GG_m$, we have for all~$T\geq \max(m,T^1_{\delta/3,M})$ and all $y \in [\theta T - M\sqrt{T}, \theta T + M\sqrt{T}]$ that
$$
\left(1-\frac{\delta}{3}\right)g_x(t_0,W_{t_0},I_{t_0}) \leq \frac{P_{W_{t_0},T-t_0}(y-I_{t_0})}{P_{x,T}(y)}
\leq \left(1+\frac{\delta}{3}\right)g_x(t_0,W_{t_0},I_{t_0}),
$$ by part~(b) of Corollary~\ref{l:long_limit} and the definition of~$\GG_m$.
Using these estimates for numerator and denominator in (\ref{eqn:whatestimate}) twice,
we obtain
$$
1-\delta \leq \frac{1-\delta/3}{1+\delta/3}\leq \frac{\p_x(D|I_T\le y)}{\p_x(D|I_T\le \eps T)} \leq \frac{1+\delta/3}{1-\delta/3}\leq 1 + \delta.
$$
This holds for any~$T$ with~$T\geq m$ and~$T\geq T^1_{\delta/3,M}=:T^2_{\delta,M}$.
\end{proof}

We have now set up all the necessary ingredients to give the proof of our main result.

\begin{proof}[Proof of Theorem~\ref{thmMain}]

\textit{Step 1: Relation between conditional probabilities.}

Setting~$\lambda:=\frac{1}{8\eps^2}$, let~$\eta$ be an exponential random variable with expectation~$\lambda^{-1}$ independent of~$W$. Then~$\{I_T \le \eta\}$ is  the survival event of a $\lambda$-killed Brownian motion up to time~$T$ as in Theorem~\ref{thmExp}. Let~$p_T(\D u)$ denote the distribution of~$I_T$.
Further, let~$D\in \sigma(W_s,s\ge 0)$. By independence of~$\eta$ and~$W$, we have
\begin{eqnarray}  \notag
    \p_x(D|I_T \le \eta)
    &=&  \frac{\p_x(D\cap\{I_T\le \eta\})}{\p_x(I_T \le \eta)}
\\\notag
    &=&  \int_{0}^{\infty} \frac{\lambda}{\p_x(I_T \le \eta)}\,\e^{-\lambda y} 
    \p_x(D\cap\{I_T\le y\}) \D y
\\ \label{eqn:integralrepofkilledprocdistr}
   &=&  \int_{0}^{\infty} \phi_T(y)  \p_x(D |  I_T\le y )   \D y,
\end{eqnarray}
where
\[
  \phi_T(y) :=  \frac{\lambda  \, \e^{-\lambda y} \p_x(I_T \le y)  } {\p_x(I_T \le \eta)}.
\]
Note that~$\phi_T$ is a probability density because replacing~$D$ by the entire probability space in (\ref{eqn:integralrepofkilledprocdistr}) yields
\[
  1 =  \int_{0}^{\infty} \phi_T(y)   \D y.
\]
Identity~\eqref{eqn:integralrepofkilledprocdistr} is the key relation connecting the two kinds of conditional probabilities.
\medskip

\textit{Step 2: Normal approximation.} We show that~$\phi_T$ is asymptotically normal with variance of order~$T$.
\\
First, from~\eqref{BS96}, we have
\begin{eqnarray}
  \p_x(I_T \le \eta)
  &=& \frac{1}{\sqrt{\cosh(T \gamma)}} \, \exp\left(- \frac{x^2 \gamma \sinh(T \gamma)}{2 \cosh(T\gamma)}\right)\notag\\
   &\sim& \sqrt{2} \exp\big(-T\sqrt{\lambda/2}-x^2 \sqrt{\lambda/2} \big),
   \quad \text{as } T \to\infty.
\end{eqnarray}

Second, from Proposition~\ref{prop:smallBall2}, we know
\[
   \p_x(I_T\le y) \sim \frac{4\sqrt{y}}{\sqrt{\pi} T} \exp \left(- \frac{(T+x^2)^2}{8y}\right), \qquad \text{as }  y\sim \eps T,
\]
uniformly in~$y \in [\eps T-M\sqrt{T}, \eps T+ M\sqrt{T}]$ for every~fixed~$M \gne 0$.
Now let~$y=\eps T+z$ with~$|z|=o(T^{2/3})$. Then by a Taylor expansion with three terms, we get
\[
  \frac{(T+x^2)^2}{8y}
   = \frac{T}{8\eps} + \frac{x^2}{4\eps}   - \frac{z}{8\eps^2}  + \frac{z^2}{8\eps^3 T} + o(1).
\]
Therefore, the exponential part of~$\phi_T(y)$ at~$y=\eps T+z$ is equal to
\begin{eqnarray*}
    &&  \exp\left(  T \sqrt{\frac{\lambda}{2}} + x^2 \sqrt{\frac{\lambda}{2}} -\lambda(\eps T+z) - \frac{T}{8\eps} - \frac{x^2}{4\eps}  + \frac{z}{8\eps^2}  - \frac{z^2}{8\eps^3 T} + o(1) \right)
\\
	&=&
  \exp\left( - \frac{z^2}{8\eps^3 T} + o(1) \right).
\end{eqnarray*}
Given a fixed~$M \gne 0$, the $o(1)$-term can be chosen uniformly for all~$z \in [-M\sqrt{T},M\sqrt{T}]$. 
Taking into account the non-exponential part of the asymptotics of~$\phi_T$, i.e.,
\[
  \lambda \cdot \frac{4\sqrt{y}}{\sqrt{\pi} T} \cdot \frac{1}{\sqrt{2}}
  \sim \frac{1}{2 \eps^{3/2}\sqrt{2\pi}\sqrt{T}},
\]
we obtain the normal approximation
\begin{equation} \label{normal_appr}
  \phi_T(y) \sim \frac{1}{\sqrt{2\pi \cdot 4\eps^3 T}} \exp\left( - \frac{z^2}{2 \cdot 4\eps^3 T} \right), 
   \qquad y=\eps T+z,\ |z| = o(T^{2/3}),
\end{equation}
which again is uniform in~$z \in [-M\sqrt{T}, M\sqrt{T}]$ for every~fixed~$M \gne 0$.
\medskip

\textit{Step 3: Evaluation for ``good'' events.}
Fix~$x\in\R$, $t_0>0$, and~$\delta\in (0,1)$.
Further, fix~$m\in\N$ and let $D\in\sigma(W_s,s \in [0,t_0])$ be such that~$\p_x(D) \gne 0$ and 
$D \subseteq \GG_m$. (The latter means that~$D$ is a ``good'' event for our purpose.)
Given~$M \gne 0$, let now~$T\geq \max(m,T^2_{\delta,M})$ with~$T^2_{\delta,M}$ from part~(b) of Lemma~\ref{lem:Dconnect}. Then we know from~\eqref{all_yprime} that
\begin{equation} \label{eqn:decisive}
     \p_x(D|I_T\le y)\le  (1+\delta)  \p_x(D| I_T\le \eps T)
\end{equation}
for all $y\in[\eps T - M \sqrt{T},\eps T - M \sqrt{T}]$.
We split the integral in~\eqref{eqn:integralrepofkilledprocdistr} and use~\eqref{eqn:decisive} to obtain
\begin{eqnarray*}  
 && \p_x(D|I_T \le \eta) 
\\
 &\le& 
    \int_0^{\eps T - M \sqrt{T}} \phi_T(y) \D y
  + (1+\delta) \int_{\eps T - M \sqrt{T}}^{\eps T + M \sqrt{T}} \phi_T(y) \D y \ \p_x(D|I_T \le \eps T)  
  +  \int_{\eps T + M \sqrt{T}}^\infty \phi_T(y) \D y
\\ 
  &\le& 1-\Phi_T(M) + (1+\delta) \Phi_T(M)  \p_x(D|I_T \le  \theta T ), 
\end{eqnarray*}
where $\Phi_T(M):=  \int_{\eps T - M \sqrt{T}}^{\eps T + M \sqrt{T}} \phi_T(y) \D y$.
This is equivalent to
\[
    \p_x(D|I_T \le \eps T)  \ge \frac{1}{ (1+\delta)\Phi_T(M)} \ 
    \left[  \p_x(D|I_T \le \eta) - (1-\Phi_T(M)) \right].
\]
For each~$M>0$, the normal approximation~\eqref{normal_appr} guarantees the existence of a finite positive limit 
$\Phi(M):=\lim_{T\to\infty} \Phi_T(M)$. Moreover, $\lim_{M\to\infty} \Phi(M) = 1$. Taking the limit in~$T$, we obtain
\[
  \liminf_{T\to\infty}  \p_x(D|I_T \le \eps T) \ge  \frac{1}{ (1+\delta)\Phi(M)} 
  \left[   \liminf_{T\to\infty} \p_x(D|I_T \le \eta) - (1-\Phi(M)) \right].
\]
Then, taking the limit in~$M$ and noting $\lim_{M\to\infty} \Phi(M) = 1$, we deduce
\begin{equation} \label{eqn:ergfirstpart2}
  \liminf_{T\to\infty}  \p_x(D|I_T \le \eps T) \ge  \frac{1}{1+\delta} \  \liminf_{T\to\infty} \p_x(D|I_T \le \eta),
\end{equation}
which holds for any~$\delta\in (0,1)$, any~$m\in\N$, and any~$D\subseteq \GG_m$.
\newpage

\textit{Step 4: Evaluation for general events.}
To simplify the notation, let us w.l.o.g.~assume that~$W$ is the canonical process on the Wiener space~$\CC([0,\infty))$ so that~$\p_x$ denotes the law of a Brownian motion started at~$x$.
According to Theorem~\ref{thmExp}, the conditional probability measures~$\p_x(\cdot\,|I_T \le \eta)$ converge weakly to the law~$\p^\ast_x$ of the proposed limiting process.
Now consider a ``general'' event $D\in\sigma(W_s,s \in [0,t_0])$ with~$\p_x(D) \in (0,1)$ which is a $\p^\ast_x$-continuity set, i.e., a set with~$\p^\ast_x(\partial D)=0$, where~$\partial D$ denotes the boundary of~$D$. Then we have $\lim_{T \to \infty} \p_x(D|I_T \le \eta) = \p_x^\ast(D)$. 
\\
Let us now fix~$m\in\N$. Recalling that~$\p^\ast_x$ is the law of an Ornstein--Uhlenbeck process, we have~$\p^\ast_x(W_{t_0}^2=m^{1/4})=0$ and~$\p^\ast_x(I_{t_0}^2 = m^{1/2})=0$. Consequently, $\GG_m$, as defined in~\eqref{eqn:defgm}, is a $\p^\ast_x$-continuity set. This implies that~$D\cap \GG_m$ is a $\p^\ast_x$-continuity set, too. Using~$D\cap \GG_m \subseteq\GG_m$, we may apply~\eqref{eqn:ergfirstpart2} to get
$$
    \liminf_{T\to\infty} \p_x(D|I_T\leq \eps T) \geq \liminf_{T\to\infty} \p_x(D\cap \GG_m|I_T\leq \eps T)  
    \geq \frac{1}{1+\delta}\ \p^\ast_x(D\cap G_m).
$$
Noting~$D\cap\GG_m \nearrow D$, as~$m\nearrow \infty$, we deduce
$$
    \liminf_{T\to\infty} \p_x(D|I_T\leq \eps T) \geq \frac{1}{1+\delta} \ \p^\ast_x(D).
$$
We can now let~$\delta\searrow 0$ to obtain
$$
    \liminf_{T\to\infty} \p_x(D|I_T\leq \eps T) \geq \p^\ast_x(D).
$$
Applying this to~$D^c$ and combining the two relations, we get
$$
     \lim_{T\to\infty} \p_x(D|I_T\leq \eps T) = \p^\ast_x(D).
$$
If $D\in\sigma(W_s,s \in [0,t_0])$ is a $\p^\ast_x$-continuity set with~$\p_x(D) \in \{0,1\}$, we trivially have $$\lim_{T\to\infty} \p_x(D|I_T\leq \eps T) = \p_x(D) = \lim_{T\to\infty} \p_x(D|I_T \le \eta) = \p_x^*(D).$$ Thus $\p_x(\cdot\,|I_T\leq \eps T)$ converges weakly to~$\p^\ast_x$ when restricting the measures to $\sigma(W_s,s \in [0,t_0])$. As this holds for any~$t_0>0$, we are done (cf.~Theorem~5 of~\cite{Whi70}).
\end{proof}

\begin{rem} \label{rem:referee}
An anonymous referee of this manuscript suggested a direct proof of Theorem~\ref{thmMain} (i.e., not appealing to Theorem~\ref{thmExp}) along the following lines, which are ``inspired by the proof of Theorem 3.6 in \cite{RVY06}''. We are very grateful for this suggestion and believe that the idea of the alternative proof should be presented here, too.

Let $t_0>0$ be fixed and consider an event $D\in\sigma(W_s,s \in [0,t_0])$. Then
\begin{equation}
\label{eqn:refereeapproach}
\p_x( D | I_T \leq \theta T) = \frac{\p_x( D \cap\{ I_T \leq \theta T\})}{\p_x( I_T \leq \theta T)} = \E_x[ \1_{D} \varphi(W_{t_0},I_{t_0},T)],
\end{equation}
where we used the Markov property in the last step and define $\varphi$ by
$$
\varphi(z,y,T):=\frac{\p_z( I_{T-t_0} \leq \theta T-y)}{\p_x( I_T \leq \theta T)}.
$$
Now, using Corollary~\ref{l:long_limit}, one obtains that
$$
\lim_{T\to\infty} \varphi(z,y,T) = \exp\left(\frac{t_0}{4\theta}-\frac{z^2}{4\theta}+ \frac{x^2}{4\theta}-\frac{y}{8\theta^2}\right)
$$
and thus, as $T\to\infty$,
\begin{equation} \label{eqn:refereeapproach2}
\varphi(W_{t_0},I_{t_0},T) \to \exp\left(\frac{t_0}{4\theta}-\frac{W_{t_0}^2}{4\theta}+\frac{x^2}{4\theta}-\frac{I_{t_0}}{8\theta^2}\right) =: \Lambda
\end{equation}
in an appropriate sense. Using that $W$ is a Brownian motion starting at $x$ under $\p_x$, we can also use that, by Ito's formula,
$$
W_{t_0}^2 = x^2 + 2 \int_0^{t_0} W_s \D W_s + t_0.
$$
This simplifies $\Lambda$ to
$$
\Lambda=\exp\left( -\frac{1}{2\theta} \int_0^{t_0} W_s \D W_s - \frac{1}{8\theta^2} \int_0^{t_0} W_s^2 \D s\right).
$$
Plugging this into \eqref{eqn:refereeapproach} shows that 
$$
\lim_{T\to\infty} \p_x( D | I_T \leq \theta T) = \E_x\left[ \1_{D} \exp\left( -\int_0^{t_0} \frac{W_s}{2\theta} \D W_s - \frac{1}{2} \int_0^{t_0} \left( \frac{W_s}{2\theta}\right)^2 \D s\right)\right],
$$
and Girsanov's theorem tells us that the right hand side is the law of an Ornstein--Uhlenbeck process satisfying the SDE $\D X_t = \D B_t - \frac{X_t}{2\theta} \D t$ on $[0,t_0]$.

We are convinced that this idea can also be made rigorous. The problem, as in our proof, is that the estimates for~$\varphi(W_{t_0},I_{t_0},T)$ required in~\eqref{eqn:refereeapproach2} fail for parts of the probability space, namely when $W_{t_0}$ or $I_{t_0}$ attain atypically large values. This should be handled in the same way as in our Lemma~\ref{lem:Dconnect}. 
\end{rem}

\section{Proof of Theorem~\ref{thmMainScaled} }
\label{sec:Scaling}

The scaling arguments we need are summarized in the following proposition.
\begin{prop} \label{p:scaling}
Let $T,q>0$ and $x\in\R$. Then
\begin{equation*}
      \p_x\left( \frac{1}{q}\, (W_{q^2 t})_{t \ge 0} \in \cdot \,\, \bigg\abs  \, I_T \le q^2 T\right)
      =
       \p_{x/q} \left( (W_t)_{t\geq 0} \in \cdot \,\, \big\abs  \, I_{T/q^2} \le T/q^2 \right).
\end{equation*}
\end{prop}

\begin{proof} 
Let $(\bar W_t)_{t\ge 0}$ be a standard Brownian motion and recall the scaling property: $(\bar W_{q^2 t})_{t\geq 0}$ has the same distribution as $(q \bar W_{t})_{t\geq 0}$.
Then we have
\begin{eqnarray*}
\p_{x/q} \left( (W_t)_{t\geq 0} \in \cdot\,\, \big\abs  \, I_{T/q^2} \le T/q^2 \right)  
    &=&
    \frac{\p\left((x/q+\bar W_t)_{t\geq 0} \in \cdot \, , \int_0^{T/q^2} (x/q+\bar W_s)^2 \D s   \le T/q^2  \right)}
    {\p\left(  \int_0^{T/q^2} (x/q+\bar W_s)^2 \D s\le T/q^2  \right)}
\\
     &=&
      \frac{\p\left( \big(\frac{1}{q}(x + \bar W_{q^2t})\big)_{t\geq 0}\in \cdot \, , \int_0^{T/q^2} (x+ \bar W_{q^2s})^2 \D s\le T \right)}
     {\p\left(  \int_0^{T/q^2} (x + \bar W_{q^2s})^2 \D s  \le T \right)}
\\
     &=&    
     \frac{\p\left(   \big( \frac{1}{q}(x + \bar W_{q^2t})\big)_{t\geq 0}  \in \cdot \,, \int_0^{T} (x+\bar W_s)^2 \D s   \le q^2 T  \right)}
     {\p\left(  \int_0^{T} (x+\bar W_s)^2 \D s   \le q^2 T  \right)}
 \\
     &=&  
     \p_{x} \left( \left(\frac{1}{q} W_{q^2 t}\right)_{t\geq 0} \in \cdot \,\, \bigg\abs  \, I_{T} \le q^2 T \right),
\end{eqnarray*}
as claimed.
\end{proof}

\begin{proof}[Proof of Theorem~\ref{thmMainScaled}]
Applying Proposition~\ref{p:scaling} with $x=0$ and $q=\sqrt{\eps_T}$, we obtain
\[
    \p_0\left(  \frac{1}{\sqrt{\eps_T}} (W_{\eps_T t})_{t \ge 0} \in \cdot \,\, \bigg\abs  \, I_T \le \eps_T T\right)
      =
       \p_{0} (  (W_t)_{t\geq 0} \in \cdot \,\, \abs  \, I_{T/\eps_T} \le T/\eps_T ).
\]
 Note that the right hand side is a special case of the conditional measure handled in Theorem~\ref{thmMain} (with $x=0$, $\eps=1$, and $T'=T/\eps_T$). 
By the assumption of Theorem~\ref{thmMainScaled}, $T/\eps_T\to\infty$, hence, by Theorem~\ref{thmMain}, the probability measures
\[
   \p_{0} ( (W_t)_{t\geq 0} \in \cdot \,\, \abs  \, I_{T/\eps_T} \le T/\eps_T )
\]
converge to the distribution of the proposed Ornstein--Uhlenbeck process.
\end{proof}

\section{Remarks on possible extensions} \label{sec:Q}

\paragraph{Bounded running $L_2$-norm.} Instead of looking at the conditioned measures
$$
\p_x\left( (W_t)_{t\geq 0} \in \cdot\,\, \left|\, \int_0^T W_s^2 \D s \leq \eps T \right. \right)
$$
one can consider the more stringent conditioning
$$
\p_x\left( (W_t)_{t\geq 0} \in \cdot\,\, \left|\, \int_0^t W_s^2 \D s \leq \eps t,\, \forall t \in [0,T] \right.\right),
$$
for~$x^2 < \theta$ as~$T \to \infty$. This line of research is in the spirit of~\cite{BB11} and~\cite{kolbsavov16}.

\paragraph{General potentials.} It would be  natural to extend the present result by conditioning Brownian motion on $\{I^{(Q)}_T \le \eps T\}$ with $I^{(Q)}_T:=\int_0^T Q(W_s)\D s$ and rather general rate function~$Q$. Under fairly general assumptions on~$Q$, an analog of Theorem~\ref{thmExp} can be obtained (see, e.g., \cite{RVY06}).
As for the corresponding extension of Theorem~\ref{thmMain}, we believe that, under essentially the same assumptions on~$Q$, the following is true.

\begin{conj}
Let~$\eps \gne 0$.
As~$T \to \infty$, the probability measures 
$$
   \p_x\big(  (W_t)_{t\geq 0} \in \cdot \,\, \big\abs\, I^{(Q)}_T \le \eps T\big)
$$ 
converge weakly on~$\CC([0,\infty))$ to the law of a diffusion process~$(X_t)_{t \ge 0}$ satisfying 
$$
  X_0=x, \qquad \D X_t = \D B_t + \frac{\phi'}{\phi}(X_t)\, \D t, \quad t \ge 0,
$$ 
where~$(B_t)_{t \ge 0}$ is a standard Brownian motion and $\phi=\phi^{(\beta)}$ is the eigenfunction of the  Schr\"odinger boundary value problem
\[
    \phi''(x)- \beta Q(x)\phi(x) = -\lambda \phi(x)
\]
corresponding to its smallest eigenvalue $\lambda=\lambda^{(\beta)}$.
The parameter $\beta=\beta(\eps)$ depends on $\eps$ in a rather implicit way.
\end{conj}

This line of research requires more technicalities because many of the explicit formulas available in the quadratic case have no direct analogs for general $Q$. It is therefore left for subsequent investigations.

\bigskip
\textbf{Acknowledgements.} We are grateful to the anonymous referee for careful reading of the manuscript, pointing out several typos, and for suggesting the nice alternative proof sketched in Remark~\ref{rem:referee}.

\end{document}